\documentclass[12pt,reqno]{amsart}
\usepackage{amsmath,amssymb,amsbsy,amsfonts,latexsym,amsopn,amstext,cite,
                                               amsxtra,euscript,amscd,bm}
\usepackage{url}

\usepackage{mathrsfs}
\usepackage{enumitem}
\usepackage{color}
\usepackage[colorlinks,linkcolor=blue,anchorcolor=blue,citecolor=blue,backref=page]{hyperref}
\usepackage{color}
\usepackage{graphics,epsfig}
\usepackage{graphicx}
\usepackage{float}
\usepackage{epstopdf}
\hypersetup{breaklinks=true}
\usepackage{hyperref}
\usepackage[T1]{fontenc}
\usepackage[latin1]{inputenc}
\usepackage{latexsym}
\usepackage[arrow, matrix, curve]{xy}

\usepackage{bm}
\usepackage{accents}
\usepackage{stmaryrd}
\usepackage{amssymb}
\usepackage{amsmath}
\usepackage{verbatim}
\usepackage{amsthm}
\usepackage{enumerate}
\usepackage{mathrsfs}
\usepackage{bbm}
\usepackage[inner=3cm,outer=3cm]{geometry}
\def \bphi{\bm{\varphi}}

\def\({\left(}
\def\){\right)}

\def\e{{\mathbf{\,e}}}

\usepackage[np]{numprint}
\npdecimalsign{\ensuremath{.}}

\usepackage{bibentry}

\usepackage[english]{babel}
\usepackage{mathtools}
\usepackage{todonotes}
\usepackage{url}
\usepackage[colorlinks,linkcolor=blue,anchorcolor=blue,citecolor=blue,backref=page]{hyperref}

\usepackage[norefs,nocites]{refcheck}

\newtheorem{theorem}{Theorem}
\newtheorem*{theorem*}{Theorem}
\newtheorem{lemma}{Lemma}

\numberwithin{equation}{section}
\numberwithin{table}{section}
\numberwithin{figure}{section}

\theoremstyle{definition}

\theoremstyle{remark}

\numberwithin{equation}{section}

\newcommand{\mmod}[1]{\,\,(\text{mod}\,\,#1)}

\def\scrC{{\mathscr C}}

\def\F{\mathbb F}

\def\N{\mathbb N}

\def\Q{\mathbb Q}
\def\R{\mathbb R}

\def\T{\mathbb T}

\def\Z{\mathbb Z}

\def\alp{{\alpha}}  
\def\bet{{\beta}}
\def\gam{{\gamma}} \def\Gam{{\Gamma}} 
 \def\Del{{\Delta}}

\def\phi{{\varphi}}

\def\balp{{\boldsymbol \alpha}}

\def\bphi{{\boldsymbol \varphi}}

\def\d{{\partial}}

\def\le{\leqslant} \def\ge{\geqslant}

\def\d{{\,{\rm d}}}
\def\mmod#1{\;(\mathrm{mod}\;{#1})}

\DeclareMathOperator{\lc}{lc}

\begin{document}

\title[Weyl sums failing square-root cancellation]{Two-dimensional Weyl sums failing square-root cancellation along lines}

\author{Julia Brandes and Igor E. Shparlinski}
\address{JB: Mathematical Sciences, University of Gothenburg and Chal\-mers Institute of Technology, 412 96 G\"oteborg, Sweden}
\email{brjulia@chalmers.se}
\address{IES: Department of Pure Mathematics, University of New South Wales, Sydney NSW 2052, Australia}
\email{igor.shparlinski@unsw.edu.au}

\date{}

\begin{abstract}
	We show that a certain two-dimensional family of Weyl sums of length $P$ takes values as large as $P^{3/4 + o(1)}$ on almost all linear slices of the unit torus, contradicting a widely held expectation that Weyl sums should exhibit square-root cancellation on generic subvarieties of the unit torus. This is an extension of a result of J.~Brandes, S.~T.~Parsell, C.~Poulias, G.~Shakan and R.~C.~Vaughan (2020) from quadratic and cubic monomials to general polynomials of arbitrary degree. The new ingredients of our approach are the classical  results  of E.~Bombieri (1966) on exponential sums along a curve and R.~J.~Duffin and A.~C.~Schaeffer (1941) on Diophantine approximations by rational numbers with prime denominators.
\end{abstract}

\maketitle

\section{Introduction}

Given their central role in many number theoretic applications, it is no surprise that Weyl sums and their properties have been subject to thorough investigation over the years. For a collection $\bphi$ of linearly independent polynomials $\phi_1, \ldots, \phi_r \in \Z[X]$ with respective degrees $k_1, \ldots, k_r$ we consider 
the {\it Weyl sums} 
$$
		f_{\bphi}(\balp)= \sum_{1 \le x \le P} \e(\alp_1 \phi_1(x) + \ldots + \alp_r \phi_r(x)),
$$
where $\e(z) = \exp\(2 \pi i z\)$ and $\balp=(\alp_1, \ldots, \alp_r)$. We also write $\T = \R / \Z$ for the unit torus, and refer to the end of this section for other notational conventions we use.

Whilst it is well known that $f_{\bphi}(\balp)$ can be of order $P$ when the entries of $\balp$ lie in the neighbourhood of fractions with a small denominator, the general expectation has always been that for a ``typical'' $\balp$ one should have the upper and lower bounds
\begin{equation} \label{sqrt-bd}
	P^{1/2} \ll f_{\bphi}(\balp) \ll P^{1/2 + o(1)}.
\end{equation} 
This question has recently been investigated in work by Chen and Shparlinski~\cite{CS1}, which in particular implies that the bounds~\eqref{sqrt-bd} hold for a subset of $\balp \in \T^r$ of full Lebesgue measure whenever the polynomials $\bphi$ have a non-vanishing Wronskian~\cite[Corollary~2.2]{CS1}. A particularly strong version of this result, applicable to the situation when $\phi_j(X)=X^j$ for $1 \le j \le r$, is available in subsequent work~\cite{CKMS}, where the interested reader will also find a more comprehensive bibliography on the subject.

In practical applications it is often necessary to control the size of $f_{\bphi}(\balp)$ on linear slices of $\T^r$, where some of the $\alp_i$ are fixed to lie in some set of full measure, whereas the remaining ones range over the entire unit interval. Such situations typically arise in ``minor arcs'' situations where some, but not all, entries of $\balp$ may have a good rational approximation and thus lie in an anticipated exceptional set. This problem has recently been studied in a very general setup by Chen and  Shparlinski~\cite{CS1} (see also~\cite{CS3}), refining an approach developed by Wooley~\cite{TDW16}. Their main result~\cite[Theorem~2.1]{CS1} asserts that whenever the polynomials $\bphi$ have a non-vanishing Wronskian, then for almost all $(\alp_1, \ldots, \alp_d) \in \T^d$ one has bounds of the shape 
$$
	\sup_{\alp_{d+1}, \ldots, \alp_r \in \T} |f_{\bphi}(\alp_1, \ldots, \alp_r)| 	\ll P^{1/2 + \Gamma(d, \bphi) + o(1)}, 
$$ 
where $\Gamma(d, \bphi)$ is a non-negative function depending on the degrees of the polynomials $\bphi$, for the precise definition of which we refer to~\cite{CS1}. 
Unfortunately, even though the bound of~\cite[Theorem~2.1]{CS1} gives strong results in a number of configurations and notably implies that one can take $\Gam(d, \bphi)=0$ for all admissible $r$-tuples of polynomials when $d=r$, in many other cases the bounds it furnishes do not beat even the trivial bound. In such situations, one has to resort to the more classical methods employing bounds of Weyl or Hua type and their subsequent generalisations (see~\cite[Lemma~2.4 and Theorem~5.2]{V:HL} for the former, and \cite[Lemma~2.5]{V:HL} as well as the results of \cite[Section~14]{TDW19} for the latter).
Bounds of this nature provide also the crucial input in the work by Erdo\u{g}an and Shakan~\cite{ES}, as well as in recent work by Chen and Shparlinski~\cite{CS2} in which, motivated by some links to certain questions on classical partial differential equations, they establish upper bounds along linear slices of the exponential sum associated with pairs of polynomials $\phi_1, \phi_2$ differing by a linear term. 
Several related results have recently been obtained by Barron~\cite{Barr}. However, as these bounds use Vinogradov's mean value theorem 
(see~\cite[Theorem~1.1]{BDG} or~\cite[Theorem~1.1]{TDW19}) as their main input, which is inefficient for Weyl sums whose degree exceeds their dimension, they are inherently unable to provide bounds stronger than $O(P^{1-c_k})$ for some positive parameter $c_k$ of size $c_k \asymp k^{-2}$.

Whilst exponents of this magnitude are not believed to be sharp in general, Brandes et al.~\cite{BPPSV} have recently shown that one cannot hope 
to have $\Gam(d, \bphi)=0$ for all choices of polynomials with non-vanishing Wronskian when $d <r$. 
In particular, for the choice $\phi_1(x) = X^k + X$ and $\phi_2(X)=X^k$ with $k=2$ or $k=3$, they show in~\cite[Theorem~1.3]{BPPSV} that for all $\alp_2 \in \R \setminus \Q$ and any $\tau> 0$ there exist arbitrarily large values of $P$ for which we have the lower bound
\begin{equation} \label{BPPSV-bd}
 	\sup_{\alp  \in \T} |f_{\bphi}(\alp_1, \alp_2)| \gg P^{3/4 -\tau},
\end{equation} 
and that for almost all $\alp_2 \in \T$ this bound can be matched by a corresponding upper bound 
$$
	\sup_{\alp  \in \T} |f_{\bphi}(\alp_1, \alp_2)|\ll P^{3/4 + o(1)}.
$$

To our knowledge, this is the first indication in the literature that the expectation that~\eqref{sqrt-bd} should hold for all $\balp$ on a linear slice of $\T^r$ may be too naive. In~\cite{BPPSV} the authors speculate that the same behaviour as in~\eqref{BPPSV-bd} might continue to hold for polynomials $\phi_1(X) = X^k + X$ and $\phi_2(X)=X^k$ with $k\ge 4$. 

The goal of this paper is therefore to extend the bound in~\eqref{BPPSV-bd} to more general polynomials, allowing also for higher degrees.

\begin{theorem}
\label{thm:main} 
 	Let $\phi \in \Z[X]$ be a polynomial of degree $k \ge 2$, and set 
 	\begin{equation} \label{expsum}
	 	f(\alp_1, \alp_2) = \sum_{1 \le x \le P} \e(\alp_1 (\phi(x)+x) + \alp_2 \phi(x)).
 	\end{equation} 
 	There exists a set $\scrC \subseteq \T$ of full Lebesgue measure such that for any $\tau>0$ and all $\alp_2 \in \scrC$ there exist arbitrarily large values of $P$ for which one has the bound 
	$$
 		\sup_{\alp_1 \in \T} |f(\alp_1, \alp_2)| \gg  P^{3/4-\tau}.
	$$
\end{theorem}

Thus, whenever $\bphi = (\phi_1, \phi_2)$ is a pair of polynomials differing only by a linear term, the associated exponential sum is are substantially larger than originally anticipated on almost all linear slices of $\T$. The fact that in our result the polynomials under consideration differ only by a linear term seems to play a role, since linear exponential sums do not exhibit square root cancellation in the same manner as their cousins of higher degree do. It is therefore an interesting question to investigate whether the behaviour observed in Theorem~\ref{thm:main}  persists, perhaps in a weaker form, even when the polynomials occurring in the exponential sum differ by more than a linear term.

Unlike in~\cite{BPPSV}, our result in Theorem~\ref{thm:main} is not complemented by a corresponding upper bound. The methods presented in~\cite{BPPSV} could conceivably be adapted to provide such upper bounds even in the more general case considered in the manuscript at hand for all $\alp_2$ lying in a subset of full measure of a suitably defined set of ``major arcs''. This would be sufficient when $k \le 3$, as then the entire unit interval $\T$ can be covered by such major arcs. For higher degrees, these methods fail and we have no improvements over the existing results of~\cite{CS2}. Nonetheless, we believe that these difficulties are of a technical rather than fundamental nature, and consequently it seems likely that the exponent $3/4$ should be sharp in those cases also.

Our argument is a streamlined version of that presented in~\cite[Section~8]{BPPSV}, which deals with the case of $\phi(X)=X^k$ for $k=2, 3$. 
However, we augment this approach by two classical results. Firstly, we appeal to a bound of Bombieri~\cite[Theorem~6]{Bom} on exponential sums along a curve over a finite field, and secondly we make use of a result of Duffin and Schaeffer~\cite[Theorem~I]{DuSch} which allows us to restrict to the case where the diophantine approximations we consider have a prime denominator.

\textbf{Notation.} Throughout the paper, we make use of the following conventions. When $x \in \R$ we denote by $\|x \|$ the distance from $x$ to the nearest integer. Moreover, $P$ always denotes a large positive number, and the letter $p$ is reserved for primes. We use the Vinogradov `$\ll$', `$\gg$' and equivalent Bachmann--Landau notations `$O(\cdot)$' liberally, and here the implied constants are allowed to depend on $\bphi$ and $\tau$, but never on $P$ or $\balp$.

\section{Assembling the toolbox}

\subsection{Approximations by rational  exponential  sums} 
In  our examination of the exponential sum~\eqref{expsum} we  rely heavily on our understanding of the closely related sum 
$$
	g(\alp, \gam) = \sum_{1 \le x \le P} \e(\alp x + \gam \phi(x))
$$
and its associated approximations. Indeed, it is apparent from the respective definitions of these exponential sums that 
\begin{equation} \label{f=g}
	f(\alp_1, \alp_2) = g(\alp_1, \alp_1+\alp_2).
\end{equation} 
When $\phi(X)=X^k$, the latter one of these has been studied in~\cite{BR} and~\cite{BPPSV}, but it turns out that in the situation we are mainly interested in the pure power may be replaced by a more general polynomial. For $q \in \N$, $a, c \in \Z$ and $\bet \in \R$ set 
$$
	S(q; a,c)= \sum_{x=1}^q \e \left( \frac{a x + c \phi(x)}{q} \right) \qquad \text{ and } \qquad I(\bet) = \int_0^P \e(\bet x) \d x,
$$
and recall that for non-vanishing $\bet$ we can compute
\begin{equation} \label{I-bd}
	|I(\bet)| =P  \left| \frac{\sin(\pi \beta P)}{\pi \beta P} \right| \ll \min \{P, \| \bet \|^{-1} \},
\end{equation} 
while a classical Weil bound (see, for example,~\cite[Corollary~II.2F]{Schmidt}) shows that when $p$ is prime and $c \nmid p$ one has
\begin{equation} \label{S-bd}
	S(p; a,c) \le (k-1)p^{1/2}.  
\end{equation}
We then have the following straightforward modification of~\cite[Theorem~3]{BR} or~\cite[Theorem~4.1]{V:HL}.

\begin{lemma}\label{L1}
	Let $\phi \in \Z[X]$ be a polynomial of degree $k \ge 2$. Suppose that $\gam \in \Q$ with $\gam = c/p$ in lowest terms, where $p$ is a prime number, and fix $a \in \Z$ such that $|\alp - a/p| \le (2p)^{-1}$. Set then $\bet = \alp - a/p$. In this notation we have 
	$$
		g(\alp,\gam) = p^{-1} S(p; a,c) I(\bet) + O(p^{1/2} \log p).
	$$ 
\end{lemma}

\begin{proof}
	Just like in the proof of~\cite[Theorem~4.1]{V:HL}, we sort the variables into residue classes, which we then encode in terms of exponential sums. Thus
	$$
		g(\alp, \gam)=\frac{1}{p} \sum_{b=1}^p S(p; a+b, c) f(\bet-b/p,0).
 	$$
 	By~\cite[Lemma~4.2]{V:HL} we have $f(\bet-b/p,0) = I(\bet-b/p)+O(1)$, so that together with~\cite[Lemma~2.2]{BPPSV} we find that 
 	$$	 	
	 	g(\alp, \gam)=\frac{1}{p} \sum_{b=1}^p S(p; a+b, c) I(\bet-b/p) + O(p^{1/2}).
 	$$
 	Since $c \nmid p$, it follows upon deploying~\eqref{I-bd} and~\eqref{S-bd} that 
  	$$	 	
 		g(\alp, \gam) - p^{-1} S(p; a,c) I(\bet) \ll p^{-1/2 } \sum_{b=1}^{p-1} \|\bet - b/p \|^{-1} \ll p^{1/2} \log p,
 	$$ 
 	where in the last step we use that 
	$$
		\| \bet-b/p \|  \ge (2p)^{-1}
	$$ 
	for all $b \not \equiv 0 \mmod p$. This completes the proof. 
\end{proof}

\subsection{A lower bound on rational exponential sums} 

Our second main tool shows that the complete exponential sum $S(p; a,c)$ cannot be smaller than $p^{1/2}$ too often. It is useful to denote the leading coefficient of $\phi$ by $\lc(\phi)$. 

\begin{lemma}\label{L2}
	Let $p$ be a prime satisfying $p>(2k)^4$ with $p \nmid\lc(\phi)$, and let $c \in \Z$ with $p \nmid c$. Then there exists $a \in \Z$ with $p \nmid (a+c)$ such that 
	$$
		S(p; a, a+c) \ge \tfrac13 p^{1/2}.
	$$
\end{lemma} 

\begin{proof}
	When $k=2$, the desired result follows from classical bounds on Gauss sums, so it is sufficient to consider the case when $k \ge 3$. By averaging and shifting the variable of summation, the result  follows if we can show that 
	\begin{equation}
\label{eq:2nd Mom}
		\sum_{a=1}^{p-1} |S(p; a-c,a)|^2 \ge  \tfrac{1}{3}  p^2
	\end{equation}
	for all primes $p > (2k)^4$ not dividing $\lc(\phi)$.

		We begin by noting that 
	\begin{align*}
		\sum_{a=1}^{p-1} |S(p; a-c,a)|^2 & = p \sum_{\substack{m,n = 1 \\ \phi(m)+m \equiv \phi(n)+n \mmod p}}^p \e \left(\frac{c(m-n)}{p}\right) - \left| \sum_{m=1}^p \e \left(\frac{cm}{p}\right) \right|^2.
	\end{align*}
	The second sum vanishes, and in the first one we make the change of variables $n=m-h$ and isolate the term corresponding to $h=0$. Hence
	\begin{equation} \label{pre-Bom}
		\sum_{a=1}^{p-1}  |S(p; a-c,a)|^2  = p^2 + p \sum_{m=1}^p \sum_{\substack{h=1 \\ \Del(m,h) \equiv 0 \mmod p}}^{p-1} \e(ch), 
	\end{equation}
	where we put 
	$$
		\Del(m,h)=(\phi(m+h)-\phi(m) + h)/h.
	$$ 
	Upon re-inserting in the term corresponding to $h=0$ and noting that all exponential sums in question take real values we discern that 
	\begin{align*}
	\sum_{m=1}^p  \sum_{\substack{h=1 \\ \Del(m,h) \equiv 0 \mmod p}}^{p-1} \e(ch) 
	&= \sum_{\substack{m,h=1 \\ \Del(m,h) \equiv 0 \mmod p}}^{p} \e(ch) - \sum_{\substack{m=1 \\ \Del(m,0) \equiv 0 \mmod p}}^p  1 \\
	&\le  \sum_{\substack{m,h=1 \\ \Del(m,h) \equiv 0 \mmod p}}^{p} \e(ch).
	\end{align*}
	If $k \ge 2$, then $\Del(X,Y)$ is a nontrivial polynomial in two variables of degree exactly $k-1$, so the congruence 
	$$
		\Del(m,h) \equiv 0 \mmod p
	$$
	defines a curve over the finite field $\F_p$.   Furthermore,  if $k> 1$, then $\Del(X,Y)$ is a nontrivial polynomial of degree exactly $k-1$ with respect to $X$ with the leading monomial $ k \lc(\phi)X^{k-1}$. Thus for $p> k$ and $p \nmid\lc(\phi)$  the variable  $h$ is not constant along this curve. 
	We may therefore apply~\cite[Theorem~6]{Bom} and find that 
	\begin{equation}\label{Bom-bd}
		\sum_{\substack{m,h=1 \\ \Del(m,h) \equiv 0 \mmod p}}^{p} \e(ch) 
						 \le \((k-1)^2 +2(k-1) -3\) \sqrt p + (k-1)^2 .
	\end{equation}  
	Under our assumption $p>(2k)^4$, for the right hand side in~\eqref{Bom-bd} we have 
	$$
 		\((k-1)^2 +2(k-1) -3\) \sqrt p + (k-1)^2 < \tfrac{2}{3}p.
	$$ 
	In view of~\eqref{pre-Bom}, we derive~\eqref{eq:2nd Mom}, which  is sufficient to establish the result. 
\end{proof}

\section{Proof of the main result}

The following result, going back to Duffin and Schaeffer~\cite{DuSch}, is a key ingredient in our arguments as it allows us to focus on those $\alp \in \T$ whose rational approximations have prime denominators. 

 \begin{lemma}\label{lem: approx a/[p}
	There is a  set $ \scrC \subseteq \T$ of full   Lebesgue measure such that for any $\alpha\in  \scrC$ there are infinitely many approximations 
	$$
		\left| \alpha - \frac{a}{p}\right| <  \frac{1}{p^2}
	$$
	with $a \in \Z$ and $p$ being a prime number. 
\end{lemma}  
\begin{proof}
	This is a direct application of~\cite[Theorem~I]{DuSch}, see also the remark on top of p.~245 of that paper. 
\end{proof}

We also remark that Lemma~\ref{lem: approx a/[p} is a   special case of the Duffin-Schaeffer conjecture, recently established as a theorem by Koukoulopoulos and Maynard~\cite{KM}. 

We now have the wherewithal to embark on the proof of Theorem~\ref{thm:main}.
Fix $\tau>0$, and let $\alp_2 \in \scrC$, where $\scrC$ is as in Lemma~\ref{lem: approx a/[p}. Then we can find an arbitrarily large prime number $p$, and $a_2 \in \Z$ not divisible by $p$, that satisfy $|\alp_2 - a_2/p| \le p^{-2}$. For any fixed such $p$ satisfying $p>(2k)^4$
 and not dividing $\lc(\phi)$,  define $P$ via the relation
\begin{equation} \label{def-P}
	P^{1+\tau} = p^2.
\end{equation} 
Lemma~\ref{L2} now guarantees the existence of an integer $a_1$ with $a_1+a_2 \not\equiv 0 \mmod p$ and having the property that 
\begin{equation} \label{large-S}
	S(p; a_1, a_1+a_2) \gg p^{1/2}.
\end{equation} 
Take now $\bet_2 = \alp_2-a_2/p$ and $\bet_1 = -\bet_2$, and put $\alp_1 = a_1/p + \bet_1$. Then upon recalling that $\gam=\alp_1+\alp_2$ in~\eqref{f=g}, we see that $\gam=c/p$ with $c = a_1 + a_2 \not\equiv 0 \mmod p$, whereupon Lemma~\ref{L1} yields the relation
$$
	g(\alp_1, \gam) = p^{-1} S(p; a_1, a_1+a_2) I(\bet_1) + O(p^{1/2} \log p).
$$
Recall now our definition of $P$ from~\eqref{def-P}. Since $|\bet_1| = |\bet_2| \le p^{-2} = P^{-1-\tau}$, it follows further from~\eqref{I-bd} that $|I(\bet_1)| = P \, (1+O(P^{-2\tau}))$,
so upon inserting~\eqref{large-S} we discern that 
$$	
	g(\alp_1, \gam) \gg P p^{-1/2} \gg P^{3/4-\tau}.
$$
In the light of~\eqref{f=g} and  Lemma~\ref{lem: approx a/[p}, this establishes the desired result.

\section*{Acknowledgements} 

The authors would like to thank James Maynard for drawing our attention to a result of  Duffin and  Schaeffer~\cite[Theorem~I]{DuSch} on Diophantine approximations with prime denominators.

During the preparation of this manuscript, JB was supported by Starting Grant no.~2017-05110 of the Swedish Science Foundation (Ve\-tenskapsr{\aa}det) and IS was supported by the Australian Research Council Grant DP170100786.

\end{document}